\documentclass[11pt]{article}
\usepackage[T1]{fontenc}
\usepackage{lmodern}
\usepackage[a4paper,margin=25mm]{geometry}
\usepackage{amsmath,amssymb,amsthm,mathtools,bm}
\usepackage{microtype,booktabs,tabularx,array,enumitem}
\usepackage{graphicx,tikz}
\graphicspath{{figures/}}
\usetikzlibrary{arrows.meta}
\usepackage[numbers,sort&compress]{natbib}
\usepackage{placeins}
\usepackage{hyperref}
\hypersetup{hidelinks,pdftitle={Near-Optimal Higher-Order Oracle Complexity for Convex-Concave Minimax Optimization}}
\newtheorem{theorem}{Theorem}[section]
\newtheorem{lemma}[theorem]{Lemma}
\newtheorem{proposition}[theorem]{Proposition}
\newtheorem{corollary}[theorem]{Corollary}
\theoremstyle{definition}
\newtheorem{definition}[theorem]{Definition}
\theoremstyle{remark}

\newcommand{\R}{\mathbb R}
\newcommand{\N}{\mathbb N}
\newcommand{\cX}{\mathcal X}
\newcommand{\cY}{\mathcal Y}
\newcommand{\cZ}{\mathcal Z}

\newcommand{\eps}{\varepsilon}
\newcommand{\Gap}{\operatorname{Gap}}
\newcommand{\rtan}{r_{\mathrm{tan}}}
\newcommand{\dist}{\operatorname{dist}}
\newcommand{\Lip}{\operatorname{Lip}}
\newcommand{\op}{\mathrm{op}}
\newcommand{\diag}{\operatorname{diag}}
\newcommand{\Span}{\operatorname{span}}
\newcommand{\cJ}{\mathcal J}

\newcommand{\tTheta}{\widetilde{\Theta}}

\newcolumntype{Y}{>{\raggedright\arraybackslash}X}
\setlist{itemsep=3pt,topsep=5pt}
\allowdisplaybreaks[1]
\title{\textbf{Near-Optimal Higher-Order Oracle Complexity for Convex--Concave Minimax Optimization}}
\newcommand{\authornames}{Yanyi Li, Haihan Zhang, Chenheng Zhang, Wendao Wu, Chunyuan Zheng, Cong Fang, Haoxuan Li, Zhouchen Lin}
\newcommand{\affiliation}{Peking University}
\newcommand{\contactemail}{colombo.2026@outlook.com}
\newcommand{\mail}[1]{\href{mailto:#1}{\texttt{#1}}}
\newif\ifanonymous
\anonymousfalse
\ifanonymous
\author{Anonymous Authors}
\hypersetup{pdfauthor={Anonymous Authors}}
\else
\author{Yanyi Li$^{1,*}$\quad Haihan Zhang$^{1,*}$\quad Chenheng Zhang$^{1,*}$\\[0.20em]
Wendao Wu$^{1,*}$\quad Chunyuan Zheng$^1$\\[0.20em]
Cong Fang$^{1,\dagger}$\quad Haoxuan Li$^{1,\dagger}$\quad Zhouchen Lin$^{1,\dagger}$\\[0.60em]
\small $^1$\affiliation\\[0.35em]
\footnotesize\mail{\contactemail}\\[0.35em]
\small $^*$Equal contribution.\quad $^\dagger$Corresponding authors.}
\hypersetup{pdfauthor={\authornames}}
\fi
\date{}
\newcommand{\ResearchAgentSystem}{our laboratory's internal auto-research system}

\begin{document}
\maketitle
\begin{abstract}
For smooth convex--concave minimax optimization, the higher-order lower bound of \citet{chen2026} applies to a restricted tensor-algorithm class with prescribed regularized Taylor-model updates. We establish the same bound for arbitrary adaptive deterministic and randomized algorithms, matching, up to logarithmic factors, the upper bound of \citet{zhang2026matching}. Fix an integer $p\ge2$ and let $L_p>0$ bound the Lipschitz constant of the objective's $p$th derivative on a compact convex product domain of diameter at most $D_Z>0$. Each feasible query returns the objective value and all derivatives through order $p$. For accuracy $\epsilon>0$, set $Q_{\mathrm{tan}}=L_pD_Z^p/\epsilon$ for tangent residual and $Q_{\mathrm{gap}}=L_pD_Z^{p+1}/\epsilon$ for saddle gap. Let $T_E^{\mathrm{det}}(\epsilon)$ and $T_E^{\mathrm{rand}}(\epsilon)$ denote the high-dimensional minimax query complexities for criterion $E\in\{\mathrm{tan},\mathrm{gap}\}$, with randomized success probability at least $2/3$ on every instance. Our lower bounds and the existing upper bound give
\[
 c_pQ_E^{2/(3p-1)}\le T_E^{\mathrm{rand}}(\epsilon)
 \le T_E^{\mathrm{det}}(\epsilon)\le C_pQ_E^{2/(3p-1)}[1+\log(3+Q_E)]^{6(p-1)}
\]
for sufficiently large $Q_E$, where $c_p,C_p>0$ depend only on $p$. Thus the same accuracy exponent holds beyond tensor update rules, even for randomized queries and arbitrary feasible outputs. The proof constructs a scalar convex--concave chain with exactly flat gates that hide complete derivative information. Direct product-domain error witnesses and adaptive transcript arguments establish the lower bounds for both criteria.
\end{abstract}
\par
\medskip
{\small
\begin{samepage}
\noindent\textbf{AI Usage.}
Nearly the entire research pipeline for this paper was carried out by
\ResearchAgentSystem{}, powered by GPT-5.6 Sol. The system also conducted a
Lean-backed article audit of the resulting manuscript. The authors subsequently reviewed and approved the
mathematical claims, presentation, and formal artifacts, and take
responsibility for the final manuscript. The complete Lean audit report and the system's technical report will be made public at a later date.
\par
\end{samepage}
}

\paragraph{Publication Intent and Recognition of Independent Work.}

This preprint is released as a research output of our study of
automated scientific research, for public scrutiny and discussion.
As we do not intend to submit this manuscript to a journal or
conference, it has not undergone the full editorial refinement
and internal review we would normally require for submission.
The authors remain responsible for its contents and for correcting
any errors or inaccurate attribution.

We recognize that releasing these results may affect researchers
who have independently been pursuing the same questions, including
work that is not yet public.
If this applies to your work, we invite you to contact us at
\texttt{colombo.2026@outlook.com}.
A brief description of the overlap, together with any dated
materials you are comfortable sharing, would help us understand
its scope and chronology.
Such materials might include draft excerpts, presentation slides,
or revision records; a complete manuscript or unpublished proof
is not required to begin the conversation.

We will review these cases carefully and update citations,
acknowledgments, and accounts of independent or concurrent work
as appropriate, including revising priority claims where warranted.
Nonpublic materials will be treated confidentially and used only
to assess the overlap unless their authors authorize otherwise.
We will obtain explicit permission from the authors before
publicly discussing such materials or providing them to an
automated research system.
\section{Introduction}\label{sec:intro}
We study the oracle complexity of smooth convex--concave minimax optimization,
\[
 \min_{x\in X}\max_{y\in Y}\phi(x,y),
\]
where $X\subseteq\mathbb R^{d_x}$ and $Y\subseteq\mathbb R^{d_y}$ are known nonempty compact convex sets and $\phi$ is convex in $x$ and concave in $y$. Fix a derivative order $p\ge2$, a Lipschitz bound $L_p>0$ for the $p$th derivative of $\phi$, and a diameter bound $D_Z>0$ for $Z=X\times Y$. An exact scalar $p$-jet query returns the function value and every derivative through order $p$ at one feasible point. Our question is the smallest worst-case number of such queries needed to attain accuracy $\epsilon>0$, allowing arbitrary computation and arbitrary adaptive query rules.

First-order extragradient methods provide the classical gap guarantees for this problem \citep{korpelevich1976,nemirovski2004}, and bilinear saddle constructions establish their worst-case limitations under general first-order query rules \citep{ouyang2021}. Higher-order methods seek faster accuracy dependence by using additional local derivative information. Newton proximal extragradient and its tensor extensions give the gap-complexity exponent $2/(p+1)$ \citep{monteiro2012,bullins2022,adil2022,lin2024}. Accelerated convex tensor methods provide an important ingredient for going beyond this rate \citep{bubeck2019,carmon2022}. Minimax acceleration improves this exponent to $4/7$ for $p=2$ \citep{chen2025} and to $4/(3p+1)$ for general $p$ \citep{chen2026}. Chen et al.\ subsequently developed higher-order Halpern methods with precision exponent $1/p$ for a proximal-residual criterion in monotone variational inequalities \citep{chen2026halpern}. Zhang et al.\ obtained exponent $2/(3p-1)$ up to logarithmic factors and explicitly recorded its scalar minimax specialization \citep{zhang2026matching}. The resulting upper bound applies to the scalar saddle operator $F_\phi=(\nabla_x\phi,-\nabla_y\phi)$.

The lower-bound question requires an additional distinction. Chen et al.\ established the exponent $2/(3p-1)$ for the tensor-algorithm class in their Definition~5.1 \citep{chen2026}. Its model centers lie in spans of previous iterates, and new iterates are generated by specified regularized Taylor-model steps, allowing both simultaneous and alternating updates. Their balanced chain shows that this broad family of methods cannot improve the exponent. An arbitrary scalar-jet algorithm, however, can choose its next query from the full transcript without obeying those update rules. The missing lower bound concerns this unrestricted information model, including random queries and an output that was never queried. The unrestricted operator lower bound in \citet{zhang2026matching} addresses the larger class of general monotone maps; hardness within the scalar convex--concave subclass requires a scalar construction.

We supply that construction. For every adaptive deterministic algorithm and every adaptive randomized algorithm, there is a smooth convex--concave instance on a product of Euclidean balls that forces the same accuracy exponent. Randomized instances are fixed independently of the realized random seed, and the lower bound rules out success probability $2/3$ on every instance below the stated budget. Let $E=\mathrm{tan}$ denote tangent residual, namely the distance of $-F_\phi(z)$ to the normal cone of $Z$, and let $E=\mathrm{gap}$ denote the primal--dual saddle gap. Write $T_E^{\mathrm{det}}$ and $T_E^{\mathrm{rand}}$ for the corresponding minimax query complexities, and set $Q_{\mathrm{tan}}=L_pD_Z^p/\epsilon$ and $Q_{\mathrm{gap}}=L_pD_Z^{p+1}/\epsilon$. Our lower bound and the upper bound of \citet{zhang2026matching} give
\begin{equation}\label{eq:intro-matching}
 T_E^{\mathrm{det}}(\epsilon)=\tTheta_p\!\left(Q_E^{2/(3p-1)}\right),
 \qquad
 T_E^{\mathrm{rand}}(\epsilon)=\tTheta_p\!\left(Q_E^{2/(3p-1)}\right),
\end{equation}
for sufficiently large $Q_E$, where $\widetilde\Theta_p$ suppresses logarithmic factors and constants depending only on the fixed order $p$. This closes the high-dimensional scalar minimax oracle complexity question up to logarithmic factors for both criteria and both algorithm classes. The upper algorithm is inherited; the new ingredient is the unrestricted scalar lower bound. \hyperref[tab:progress]{Table~\ref*{tab:progress}} distinguishes these two roles and the scope of the earlier results.

\begin{table}[t]
\centering\small
\setlength{\belowcaptionskip}{8pt}
\caption{Accuracy dependence of scalar minimax oracle bounds at fixed smoothness, geometry, and derivative order $p\ge2$, except where $p=2$ is specified. The notation $\widetilde O$, $\widetilde\Omega$, and $\widetilde\Theta$ suppresses powers of $\log(1/\epsilon)$. The acceleration bounds of \citet{chen2025,chen2026} use their stated lower-order smoothness inputs. In particular, \citet[Theorem~4.1]{chen2026} uses a gradient Lipschitz bound $L_1$ and residual tolerance $\epsilon_{\rm tan}\le L_p\min\{D_X^p,D_Y^p\}$, where $D_X,D_Y$ are the block diameters; the gap specialization takes $\epsilon_{\rm tan}=\epsilon/D_Z$. These entries compare accuracy exponents under those assumptions. Lower bounds allow the dimension to grow with the query budget. The final characterization combines our lower bound with the upper bound of \citet{zhang2026matching}.}
\label{tab:progress}
\renewcommand{\arraystretch}{1.18}
\setlength{\tabcolsep}{4pt}
\begin{tabularx}{\textwidth}{@{}>{\raggedright\arraybackslash}p{0.30\textwidth}>{\centering\arraybackslash}p{0.19\textwidth}>{\centering\arraybackslash}p{0.19\textwidth}Y@{}}
\toprule
Result & Upper bound & Lower bound & Criterion and scope\\
\midrule
Higher-order extragradient \citep{bullins2022,adil2022,lin2024}
 & $O(\epsilon^{-2/(p+1)})$ & --- & Gap\\
Second-order acceleration \citep{chen2025}, $p=2$
 & $\widetilde O(\epsilon^{-4/7})$ & --- & Gap\\
Minimax acceleration and balanced chain \citep{chen2026}
 & $\widetilde O(\epsilon^{-4/(3p+1)})$ & $\Omega(\epsilon^{-2/(3p-1)})$ & Residual and gap; lower bound for tensor algorithms\\
Convex embedding \citep{arjevani2019}
 & --- & $\Omega(\epsilon^{-2/(3p+1)})$ & Residual and gap; deterministic jets\\
Randomized embedding \citep{garg2021}
 & --- & $\widetilde\Omega(\epsilon^{-2/(3p+1)})$ & Residual and gap; randomized jets\\
Scalar VI specialization \citep{zhang2026matching}
 & $\widetilde O(\epsilon^{-2/(3p-1)})$ & --- & Residual and gap\\
\midrule
\textbf{This paper} with \citep{zhang2026matching}
 & \multicolumn{2}{c}{$\tTheta_p(\epsilon^{-2/(3p-1)})$} & Both criteria; arbitrary deterministic and randomized jets\\
\bottomrule
\end{tabularx}
\end{table}

The main obstacle is to hide information while preserving scalar convex--concave structure. A general monotone chain need not be the saddle gradient of a scalar potential: the sign-adjusted Jacobian must be symmetric. We pair primal and dual directions so that the saddle signature acts as a reversal matrix, making a directed shift compatible with this symmetry. Exactly flat smooth gates then conceal unrevealed directions from the function value and every returned derivative. The modules act on orthogonal subspaces, which preserves a dimension-independent highest-order smoothness bound.

Two further features give the full oracle lower bound. First, a direct product-domain witness controls saddle gap at the same budget exponent as tangent residual, with a separate boundary argument accounting for normal cones. Second, the directions are hidden from arbitrary transcripts: a resisting-frame argument handles deterministic algorithms, while a truncated-transcript coupling handles randomized algorithms with a single fixed instance. Against $N$ queries, the constructions use $2N+2$ dimensions in each block in the deterministic case and $O_p((N+1)^3\log(N+2))$ dimensions in the randomized case. These are high-dimensional worst-case statements, with arbitrary feasible outputs.

\section{Related Works}\label{sec:related}
\paragraph{Extragradient, tensor models, and optimism.}
Extragradient and Mirror Prox established the first-order approach to monotone saddle problems \citep{korpelevich1976,nemirovski2004}. Newton proximal extragradient extends this approach through locally regularized second-order models \citep{monteiro2012}. HigherOrderMirrorProx, line-search-free tensor methods, and Perseus develop higher-order model-based gap guarantees \citep{bullins2022,adil2022,lin2024}. Reduced-gradient methods extend the higher-order approach to composite variational inequalities \citep{nesterov2023}. Generalized optimistic methods provide another route from proximal approximations to higher-order saddle algorithms \citep{jiang2025optimistic}; adaptive second-order optimistic methods also address unknown smoothness and the cost of model computations \citep{jiang2024adaptive}. These works study useful algorithmic structures. Our oracle model permits all such internal computations while charging each query of the unknown scalar objective.

\paragraph{Accelerated upper bounds and strong solution criteria.}
The minimax acceleration of \citet{chen2025,chen2026} combines primal and dual regularization with accelerated convex subproblems. Its guarantees use the stated lower-order smoothness inputs and accuracy regime. The convex acceleration behind these subproblems has a complementary history: \citet{bubeck2019} obtain near-optimal higher-order rates, and \citet{carmon2022} remove the logarithmic overhead of solving the implicit Monteiro--Svaiter step. For strong VI solutions, \citet{diakonikolas2020} connect Halpern iteration, nonexpansive mappings, and proximal reductions to obtain near-optimal first-order guarantees. These two lines supply the acceleration and strong-solution perspectives used by subsequent higher-order methods. Halpern-NPE and Halpern-ATM improve strong-solution precision dependence for general monotone VIs \citep{chen2026halpern}. Their original proximal-residual formulation should be distinguished from the same-point tangent certificate used here. Zhang et al.\ establish a certifying higher-order VI method and a scalar specialization with exponent $2/(3p-1)$ up to logarithms \citep{zhang2026matching}. We use \citet[Theorem~3.3 and Corollary~3.4]{zhang2026matching} as the upper-bound input and give the scalar oracle and normal-cone transfer explicitly in \hyperref[prop:upper]{Proposition~\ref*{prop:upper}}. Recently, \citet{zhang2026optimal} developed an accelerated Halpern method for monotone VIs that attains the same exponent without logarithmic overhead. Their result can further remove the logarithmic factors from the upper side of the scalar minimax characterization through the same oracle reduction.

\paragraph{Convex lower bounds and tensor-restricted saddle chains.}
Convex minimization embeds in the scalar saddle class by taking an objective independent of the dual variable. The higher-order deterministic lower bound of \citet[Theorem~3]{arjevani2019} therefore supplies exponent $2/(3p+1)$. \citet{agarwal2018} develop randomized higher-order convex lower bounds, and \citet[Theorems~1 and~3]{garg2021} obtain the matching convex exponent up to logarithms using smooth information-hiding constructions. For saddle problems, \citet{adil2022,lin2024} analyze lower bounds under structured update assumptions. \citet[Sections~5.1--5.2]{chen2026} explain why diameter normalization matters for those comparisons and construct a balanced scalar chain with exponent $2/(3p-1)$. Their Theorem~5.2 and Lemma~5.1 apply to the tensor class of Definition~5.1. Our theorem retains their exponent and extends its algorithmic scope to unrestricted adaptive scalar queries.

\paragraph{Adaptive information hiding and scalar structure.}
Rotated hard instances and transcript arguments prevent an algorithm from bypassing a chain by choosing arbitrary coordinates. In first-order bilinear saddle problems, \citet{ouyang2021} prove lower bounds both for linear-span methods and for general deterministic methods based on first-order information. Their extension illustrates why an oracle lower bound must account for query rules beyond a chosen update template. At higher orders, a query also exposes entire derivative tensors, so the hidden directions must remain absent from all these responses. Higher-order convex constructions show how to hide complete derivative tensors from randomized queries \citep{agarwal2018,garg2021}. The flat-gate operator construction of \citet{zhang2026matching} gives unrestricted deterministic and randomized VI lower bounds. Passing from that operator class to scalar minimax requires both integrability and the primal--dual curvature signs. Our paired-coordinate construction enforces these properties simultaneously. Its direct gap witness also supplies the additional diameter factor needed to match the gap upper bound. The resulting lower bounds concern the same scalar oracle and feasible-query convention as the upper reduction.

\section{Oracle Model and Near-Optimal Complexity}
\label{sec:model}\label{u:sec:setup}

Fix an integer $p\ge2$.  For known nonempty compact convex sets $\cX\subseteq\R^{d_x}$ and $\cY\subseteq\R^{d_y}$, put $\cZ=\cX\times\cY$. For a differentiable potential, its saddle operator is $F_\phi(x,y)=(\nabla_x\phi(x,y),-\nabla_y\phi(x,y))$. The normal cone at $z$ is defined by
\begin{equation}\label{eq:normal-definition}
  N_\cZ(z)=\{v:\ \langle v,z'-z\rangle\le0\text{ for all }z'\in\cZ\}.
\end{equation}
The tangent residual and saddle gap are
\begin{equation}
  \rtan^\phi(z)=\dist\bigl(0,F_\phi(z)+N_\cZ(z)\bigr),
  \label{eq:model-tan}
\end{equation}
\begin{equation}
  \Gap_\phi(x,y)=\max_{y'\in\cY}\phi(x,y')-\min_{x'\in\cX}\phi(x',y).
  \label{eq:model-gap}
\end{equation}

\begin{definition}[Smooth convex--concave class]
For fixed $L_p,D_Z>0$, let $\mathfrak S_p(L_p,D_Z)$ be the class of triples $(\cX,\cY,\phi)$ such that
\begin{enumerate}[label=(\roman*),leftmargin=2.2em]
\item $\cX\subseteq\R^{d_x}$ and $\cY\subseteq\R^{d_y}$ are nonempty compact convex sets with $\operatorname{diam}(\cX\times\cY)\le D_Z$;
\item $\phi$ is defined and $p$ times continuously differentiable on a neighborhood of $\cX\times\cY$;
\item $\phi(\cdot,y)$ is convex on $\cX$ for every $y\in\cY$;
\item $\phi(x,\cdot)$ is concave on $\cY$ for every $x\in\cX$;
\item the highest derivative satisfies
\begin{equation}\label{eq:smoothness-definition}
\|D^p\phi(z)-D^p\phi(z')\|_{\op}\le L_p\|z-z'\|\qquad(z,z'\in\cZ);
\end{equation}
\item a saddle point exists by the compact convex--concave minimax theorem.
\end{enumerate}
The instance dimensions may depend on the query budget in the lower bounds.
\end{definition}

A complete scalar $p$-jet query at $z\in\cZ$ returns
\begin{equation}
  \cJ_p\phi(z)=\bigl(\phi(z),D\phi(z),D^2\phi(z),\ldots,D^p\phi(z)\bigr).
  \label{eq:jet}
\end{equation}
Every oracle query must lie in the known feasible set $\cZ$.

\begin{definition}[Unrestricted deterministic algorithm]
A deterministic $N$-query algorithm chooses each query $z_t\in\cZ$ as an arbitrary function of the public data and the transcript
\[
  (z_1,\cJ_p\phi(z_1),\ldots,z_{t-1},\cJ_p\phi(z_{t-1})).
\]
After at most $N$ queries it outputs an arbitrary transcript-dependent point $\widehat z\in\cZ$.  The output need not have been queried.  No linear-span, tensor-step, or update-rule restriction is imposed.
\end{definition}

\begin{definition}[Unrestricted randomized algorithm]
A randomized $N$-query algorithm is a deterministic algorithm after drawing an independent internal random seed.  For a fixed instance, its success probability is taken only over this seed.  A lower-bound instance for a randomized algorithm must be fixed independently of the realized seed.
\end{definition}

\paragraph{Public information and minimax complexity.}
The public data are $p,L_p,D_Z,\epsilon$, the dimensions, and the feasible sets. Neither the hidden frame nor any representation of the unknown objective is supplied. Computation based on the public sets and returned jets is uncharged; each new evaluation of the unknown objective's jet costs one call. Randomized policies and their outputs are assumed measurable. We take the query and output maps to be defined on all finite response histories, with feasible queries and an $N$-query cap on every history. Any policy specified only on its admissible histories can be extended by stopping with a fixed feasible point elsewhere; this convention makes the truncated simulations below well defined. We write
\[
 T^{\rm det}_{p,\mathrm{cc},E}(\eps)
 =\inf\{N\in\N_0:\ \exists A\ \forall\omega\in\mathfrak S_p(L_p,D_Z),\
 N_A(\omega)\le N,\ E_\omega(\widehat z_A)\le\eps\},
\]
where $E\in\{\mathrm{tan},\mathrm{gap}\}$ labels tangent residual or saddle gap, respectively, $E_\omega$ is the corresponding error, and $\omega$ includes the public domain and the objective. A policy is chosen from the public data before accessing the unknown objective. For $T^{\rm rand}_{p,\mathrm{cc},E}$, require an almost-sure cap of $N$ calls and $\Pr\{E_\omega(\widehat z_A)\le\eps\}\ge2/3$ for every fixed instance. Both complexities take the worst case over all finite dimensions. An empty feasible set of policies has infimum $+\infty$.

\begin{lemma}[Criterion comparison and convex embedding]\label{lem:criteria}
For every admissible instance and every $z\in\cZ$,
\begin{equation}
 \Gap_\phi(z)\le D_Z\rtan^\phi(z).
 \label{eq:criteria}
\end{equation}
If $\phi(x,y)=f(x)$, then $\Gap_\phi(x,y)=f(x)-\min_{u\in\cX}f(u)$, and one complete scalar $p$-jet of $\phi$ can be simulated with one $p$-jet of $f$.
\end{lemma}
\begin{proof}
Convexity and concavity imply, for $z'=(x',y')\in\cZ$,
\[
 \phi(x,y')-\phi(x',y)\le\langle F_\phi(z),z-z'\rangle.
\]
For any $v\in N_\cZ(z)$, the right side is at most
$\langle F_\phi(z)+v,z-z'\rangle\le D_Z\|F_\phi(z)+v\|$.
Taking the supremum over $z'$ and then the infimum over $v$ proves~\eqref{eq:criteria}. The embedding identities follow because every derivative involving $y$ vanishes.
\end{proof}

For the convex rows in \hyperref[tab:progress]{Table~\ref*{tab:progress}}, use the source's primal ball and a singleton dual set. In the distance-to-minimizer formulation of \citet{arjevani2019}, start the simulation with a query at the public center; its hard minimizer then belongs to the prescribed ball. A radius bound and a diameter bound differ only by a constant. Simulate the saddle algorithm by querying $f$ at its primal block. If the source theorem evaluates its output, an additional call evaluates the proposed output. These at most two extra calls do not change the asymptotic rate. An $\eps$-residual solver would give a $D_Z\eps$-suboptimality solver by~\eqref{eq:criteria}; substituting this tolerance gives the displayed residual scale. A dummy dual ball also works because $f$ is independent of~$y$.

\subsection{Unrestricted lower bounds}\label{sec:main}
The lower theorem is the main new result. It controls both criteria directly on the scalar class, rather than obtaining saddle hardness from a general monotone operator.
\begin{theorem}[Unrestricted scalar lower bounds]
\label{thm:main}
For every fixed $p\ge2$, there are constants $a_p,b_p,C_p>0$ with the following property.  Let $N\ge1$, $L_p>0$, and $D_Z>0$.

\begin{enumerate}[label=(\roman*),leftmargin=2.2em]
\item For every deterministic adaptive algorithm making at most $N$ feasible complete scalar $p$-jet queries, there are dimensions
\[
  d_x=d_y=2N+2,
\]
Euclidean balls $\cX\subseteq\R^{d_x}$ and $\cY\subseteq\R^{d_y}$ with $\operatorname{diam}(\cX\times\cY)=D_Z$, and a globally smooth convex--concave function $\phi$ with
\[
  \Lip(D^p\phi)\le L_p,
\]
such that the algorithm's possibly unqueried output $\widehat z$ satisfies
\begin{equation}
  \rtan^\phi(\widehat z)
  \ge
  a_p\frac{L_pD_Z^p}{(N+1)^{(3p-1)/2}},
  \qquad
  \Gap_\phi(\widehat z)
  \ge
  b_p\frac{L_pD_Z^{p+1}}{(N+1)^{(3p-1)/2}}.
  \label{eq:main-det}
\end{equation}

\item For every randomized adaptive algorithm making at most $N$ feasible complete scalar $p$-jet queries, there are dimensions
\[
  d_x,d_y\le C_p (N+1)^3\log(N+2),
\]
a product of Euclidean balls of diameter $D_Z$, and one fixed admissible function $\phi$, independent of the realized random seed, such that
\begin{equation}
 \Pr\!\left\{
  \rtan^\phi(\widehat z)
  \ge
  a_p\frac{L_pD_Z^p}{(N+1)^{(3p-1)/2}}
 \right\}\ge\frac34,
 \label{eq:main-rand-tan}
\end{equation}
and, for a possibly different fixed instance from the same family,
\begin{equation}
 \Pr\!\left\{
  \Gap_\phi(\widehat z)
  \ge
  b_p\frac{L_pD_Z^{p+1}}{(N+1)^{(3p-1)/2}}
 \right\}\ge\frac34.
 \label{eq:main-rand-gap}
\end{equation}
\end{enumerate}
\end{theorem}

\subsection{The upper reduction and matching characterization}
For either criterion, the established upper algorithm can be implemented using the same complete scalar oracle. We record the reduction to make the assumptions, output certificate, and query count in the matching statement explicit.

\begin{proposition}[Scalar realization of the higher-order VI upper bound]\label{prop:upper}
For fixed $p\ge2$, there is a constant $C_p>0$ such that every known compact convex product $Z=X\times Y$ of diameter at most $D_Z$ admits a deterministic scalar-jet policy returning an evaluated $z\in Z$ and $n\in N_Z(z)$ with $\|F_\phi(z)+n\|\le\epsilon$ in at most
\[
 C_p\left(1+\left(\frac{L_pD_Z^p}{\epsilon}\right)^{2/(3p-1)}\right)
 \left[1+\log\left(3+\frac{L_pD_Z^p}{\epsilon}\right)\right]^{6(p-1)}
\]
queries. A saddle gap at most $\epsilon$ is attainable with the same bound after replacing $L_pD_Z^p/\epsilon$ by $L_pD_Z^{p+1}/\epsilon$.
\end{proposition}
\begin{proof}
We apply \citet[Theorem~3.3]{zhang2026matching}, whose method uses complete operator jets $(G,DG,\ldots,D^{p-1}G)$ for a monotone map satisfying $\Lip(D^{p-1}G)\le L_p$ on a compact convex set of diameter at most $D$. With $Q=L_pD^p/\epsilon\ge2$, its query bound is $C_p(1+Q^{2/(3p-1)})[1+\log(3+Q)]^{6(p-1)}$. Algorithm~6 in that work returns an evaluated point with a normal certificate; the bound includes all inner-solver, regularization, and certification calls. Corollary~3.4 there records the scalar specialization. We give the oracle and normal-cone transfer below to make its application to the present model explicit. Each tensor model is constructed from a returned jet and the known domain, so computations on that model require no additional unknown-objective query.

Let $a=\Pi_Z(0)$, $V=\operatorname{lin}(Z-Z)$, and $K=Z-a\subset V$. Write $S=\diag(I_{d_x},-I_{d_y})$ and $G(u)=P_VS\nabla\phi(a+u)$, where $P_V$ is orthogonal projection. Convexity--concavity implies monotonicity of $F_\phi=S\nabla\phi$ on $Z$, hence of $G$ on $K$. Indeed, the mixed Hessian blocks cancel in the symmetric part of $S\nabla^2\phi$, and its quadratic form is nonnegative in feasible affine directions; integration along segments proves monotonicity, including at the relative boundary. For $0\le k\le p-1$,
\[
 \langle w,D^kF_\phi(z)[h_1,\ldots,h_k]\rangle
 =D^{k+1}\phi(z)[Sw,h_1,\ldots,h_k].
\]
Since $S$ is an isometry, projection and restriction give $\Lip(D^{p-1}G)\le L_p$. Thus one scalar query at $a+u$ supplies the entire operator jet at $u$, with no loss in the modulus or the query count. All queries remain feasible. A solution exists by compactness and convexity--concavity, and its distance from the public start is at most $D_Z$.

For the returned $u$ and relative normal $n_V\in N_K^V(u)$, put $z=a+u$ and $n=n_V-(I-P_V)F_\phi(z)$. Every feasible difference belongs to $V$, so $n\in N_Z(z)$ and
\[
 F_\phi(z)+n=G(u)+n_V.
\]
The stored scalar gradient supplies the ambient component. An additional endpoint query, if needed by an implementation, contributes only an additive constant. This accounts for every evaluation of the unknown objective; the same oracle data also determine the known affine regularizations.

For ratios below a fixed constant, enlarge the public smoothness bound to $\max\{L_p,2\epsilon/D_Z^p\}$ and apply the same result. This gives the displayed bound, including its additive constant. If $Z$ is a singleton, one query and $n=-F_\phi(z)$ suffice. Finally, \hyperref[lem:criteria]{Lemma~\ref*{lem:criteria}} with residual tolerance $\epsilon/D_Z$ establishes the stated saddle-gap bound.
\end{proof}

\begin{corollary}[Near-optimal scalar minimax oracle complexity]\label{cor:complexity}
Fix $p\ge2$. For $E\in\{\mathrm{tan},\mathrm{gap}\}$, let $Q_E=L_pD_Z^p/\epsilon$ when $E=\mathrm{tan}$ and $Q_E=L_pD_Z^{p+1}/\epsilon$ when $E=\mathrm{gap}$. There exist $c_p,C_p,Q_{0,p}>0$ such that, for $Q_E\ge Q_{0,p}$,
\begin{equation}\label{eq:matching}
 c_pQ_E^{2/(3p-1)}
 \le T^{\mathrm{rand}}_{p,\mathrm{cc},E}(\epsilon)
 \le T^{\mathrm{det}}_{p,\mathrm{cc},E}(\epsilon)
 \le C_pQ_E^{2/(3p-1)}[1+\log(3+Q_E)]^{6(p-1)}.
\end{equation}
The upper inequality is inherited from \citet{zhang2026matching} through \hyperref[prop:upper]{Proposition~\ref*{prop:upper}}; the lower inequality is \hyperref[thm:main]{Theorem~\ref*{thm:main}}. The constants are independent of the dimension, $L_p$, $D_Z$, and $\epsilon$.
\end{corollary}
\begin{proof}
Choose an integer $N$ so that the appropriate obstruction in \hyperref[thm:main]{Theorem~\ref*{thm:main}} strictly exceeds $\epsilon$. The deterministic output then fails the target, and the randomized failure probability is at least $3/4$, exceeding the allowed $1/3$. For sufficiently large $Q_E$, integer rounding and the term $N+1$ are absorbed into $c_p$. \hyperref[prop:upper]{Proposition~\ref*{prop:upper}} gives a uniform upper budget on every public domain; its additive constant is absorbed into $C_p$. The middle inequality follows by including deterministic policies among randomized policies.
\end{proof}

The characterization is high-dimensional: the hard dimension may increase with the budget, whereas the upper bound is dimension-independent. Each lower-bound instance is strongly convex--concave with a budget-dependent modulus; the problem class prescribes only convexity--concavity and the highest-order smoothness bound. The next two sections prove \hyperref[thm:main]{Theorem~\ref*{thm:main}}, first through geometric error witnesses and then through transcript hiding.

\section{The Hard Family and Its Error Geometry}\label{sec:hard}
The lower bound combines local information hiding with a global error obstruction. We build a smooth gate that hides an edge of a directed chain, embed the chain in a scalar convex--concave potential, and quantify the error caused by an unrevealed middle coordinate. Choosing the gate width and potential scale then turns this geometric obstruction into the desired dependence on the query budget.

\subsection{An exactly flat bilinear gate}\label{sec:gate}
The gate must retain the coupling of a bilinear edge away from the origin and make every derivative vanish near the origin. The following construction provides both properties while keeping its curvature uniformly bounded.

Choose once and for all an even function $\chi\in C^\infty(\R;[0,1])$ satisfying
\[
 \chi(s)=0\quad (|s|\le1/2),
 \qquad
 \chi(s)=1\quad (|s|\ge1).
\]
For $\tau>0$, define
\begin{equation}
 h_\tau(s)=\int_0^s \chi(u/\tau)\,du,
 \qquad
 H_\tau(s)=\int_0^s h_\tau(v)\,dv,
 \label{eq:hH}
\end{equation}
and
\begin{equation}
 \Theta_\tau(a,b)
 =\frac12\bigl(H_\tau(a+b)-H_\tau(a-b)\bigr).
 \label{eq:Theta}
\end{equation}

These two antiderivatives connect the desired flat region to a controlled approximation of the bilinear term. We record the local hiding, approximation, and smoothness estimates together because each serves a separate part of the lower-bound argument.

\begin{lemma}[Gate properties]
\label{lem:gate}
For every fixed $p\ge2$, the following hold.
\begin{enumerate}[label=(\alph*),leftmargin=2em]
\item $h_\tau$ is odd, $H_\tau$ is even, and
\[
 h_\tau(s)=H_\tau(s)=0\qquad (|s|\le\tau/2).
\]
Moreover,
\[
 0\le h_\tau'(s)\le1,
 \qquad
 |h_\tau(s)-s|\le\tau.
\]
\item If $|a|+|b|<\tau/2$, then $\Theta_\tau$ is identically zero on a neighborhood of $(a,b)$.  If $|a|+|b|\le\tau/2$, its value and all derivatives vanish at $(a,b)$ by the smooth, flat cutoff.
\item
\[
 \partial_b\Theta_\tau(a,b)
 =\frac12\bigl(h_\tau(a+b)+h_\tau(a-b)\bigr),
\]
\[
 \partial_a\Theta_\tau(a,b)
 =\frac12\bigl(h_\tau(a+b)-h_\tau(a-b)\bigr),
\]
and therefore
\[
 |\partial_b\Theta_\tau(a,b)-a|\le\tau,
 \qquad
 |\partial_a\Theta_\tau(a,b)-b|\le\tau.
\]
Also,
\[
 |\Theta_\tau(a,b)-ab|\le\tau |b|,
 \qquad
 \left|H_\tau(a)-\frac12a^2\right|\le\tau|a|.
\]
\item The eigenvalues of $\nabla^2\Theta_\tau(a,b)$ are
\[
 \chi((a+b)/\tau),
 \qquad
 -\chi((a-b)/\tau),
\]
so
\[
 \|\nabla^2\Theta_\tau(a,b)\|_{\op}\le1,
 \qquad
 0\le H_\tau''(a)\le1.
\]
\item There is a constant $K_p<\infty$ such that
\[
 \|D^{p+1}\Theta_\tau(a,b)\|_{\op}
 +|H_\tau^{(p+1)}(a)|
 \le K_p\tau^{1-p}
\]
for all $a,b$.
\end{enumerate}
\end{lemma}

\begin{proof}
All statements follow directly from \eqref{eq:hH}--\eqref{eq:Theta}.  For example,
$h_\tau'=\chi(\cdot/\tau)$, and
\[
 h_\tau(s)-s=-\int_0^s\bigl(1-\chi(u/\tau)\bigr)\,du,
\]
whose absolute value is at most $\tau$.  The Hessian is
\[
 \nabla^2\Theta_\tau(a,b)
 =\frac12
 \begin{pmatrix}
  \chi_+-\chi_- & \chi_++\chi_-\\
  \chi_++\chi_- & \chi_+-\chi_-
 \end{pmatrix},
 \qquad
 \chi_\pm=\chi((a\pm b)/\tau),
\]
which has the claimed two eigenvalues.  Every derivative of order $k\ge2$ of $h_\tau$ scales as $O_k(\tau^{1-k})$, and the last claim follows by one integration and the two linear forms $a\pm b$.
\end{proof}

\subsection{A Scalar Potential for the Directed Chain}\label{sec:metric}

We now assemble the gates into a chain whose operator is generated by a scalar potential. Pairing primal and dual frame directions makes the saddle signature act as coordinate reversal; this is the algebraic link between a directed shift and scalar integrability.

Let $n=N+1$ and $m=2n$.  Let
\[
 U=(u_1,\ldots,u_n)\in\R^{d_x\times n},
 \qquad
 V=(v_1,\ldots,v_n)\in\R^{d_y\times n}
\]
have orthonormal columns.  For $z=(x,y)$, set
\[
 a_i=\frac{\langle u_i,x\rangle+\langle v_i,y\rangle}{\sqrt2},
 \qquad
 b_i=\frac{\langle u_i,x\rangle-\langle v_i,y\rangle}{\sqrt2},
 \quad i=1,\ldots,n,
\]
and order the active coordinates as
\begin{equation}
 r=(a_1,\ldots,a_n,b_n,b_{n-1},\ldots,b_1)\in\R^{2n}.
 \label{eq:rorder}
\end{equation}
Let $W=W_{U,V}$ denote the corresponding isometric embedding, so $r=W^Tz$. Define the saddle signature $S=\diag(I_{d_x},-I_{d_y})$ and let $R_m$ be the $m\times m$ reversal matrix. These matrices satisfy
\begin{equation}
 W^TSW=R_m.
 \label{eq:signature-reversal}
\end{equation}
Indeed, $S$ exchanges the paired vectors
\[
 \frac{(u_i,v_i)}{\sqrt2}
 \quad\text{and}\quad
 \frac{(u_i,-v_i)}{\sqrt2}.
\]

Let $J$ be the forward shift, $Je_j=e_{j+1}$ for $j<m$ and $Je_m=0$. Then $R_mJ=J^TR_m$, so $R_mJ$ is symmetric. Thus a directed shift can arise from a scalar gradient after multiplication by the reversal metric. The paired basis realizes precisely that metric without rotating the primal and dual feasible balls into one another. \hyperref[fig:mechanism]{Figure~\ref*{fig:mechanism}} illustrates the chain and its local hiding region.

Fix
\begin{equation}
 q=1-\frac1{4n}.
 \label{eq:q}
\end{equation}
Define
\begin{equation}
 \psi(r)
 =c r_{2n}
 +q\sum_{i=1}^{n-1}\Theta_\tau(r_i,r_{2n-i})
 +qH_\tau(r_n),
 \label{eq:psi}
\end{equation}
and the scalar potential
\begin{equation}
 \phi_{U,V}(x,y)
 =\lambda\left[
  \frac12\|x\|^2-\frac12\|y\|^2-\psi(W^T(x,y))
 \right].
 \label{eq:hardpotential}
\end{equation}
The variables of the modules in \eqref{eq:psi} are pairwise disjoint:
\[
 (a_1,b_2),(a_2,b_3),\ldots,(a_{n-1},b_n),a_n,
\]
while $b_1$ appears only in the linear seed.

\begin{figure}[t]
\centering
\begin{minipage}[t]{0.56\textwidth}
\centering
\begin{tikzpicture}[x=1.05cm,y=1cm,>=Stealth,every node/.style={font=\small}]
\foreach \i in {1,2,3,4}{
 \node (a\i) at (\i,1.6) {$a_\i$};
 \node (b\i) at (\i,0) {$b_\i$};
}
\foreach \i/\j in {1/2,2/3,3/4}{
 \draw[->] (a\i)--(a\j);
 \draw[->] (b\j)--(b\i);
}
\draw[->] (a4.east) to[out=0,in=0] (b4.east);
\node[above=0.20cm] at (a1) {seed $c$};
\node[below=0.3cm] at (2.5,0) {(a) Ideal saddle-operator chain};
\end{tikzpicture}
\end{minipage}
\hfill
\begin{minipage}[t]{0.40\textwidth}
\centering
\begin{tikzpicture}[x=1.25cm,y=1.25cm,>=Stealth,every node/.style={font=\small}]
\fill[gray!15] (0,1)--(1,0)--(0,-1)--(-1,0)--cycle;
\draw (0,1)--(1,0)--(0,-1)--(-1,0)--cycle;
\draw[->] (-1.3,0)--(1.4,0) node[right] {$a$};
\draw[->] (0,-1.2)--(0,1.4) node[above] {$b$};
\node at (0,0.40) {$\Theta_\tau=0$};
\node[below] at (1,0) {$\tau/2$};
\node at (0,-1.65) {(b) Exact full-jet plateau};
\end{tikzpicture}
\end{minipage}
\caption{Two features of the hard family. (a) With $n=4$, the ideal operator carries the seed along $r=(a_1,\ldots,a_4,b_4,\ldots,b_1)$; its zero has $r_j=cq^{j-1}$. The actual gated operator approximates each shift relation within $q\tau$. (b) Each scalar gate and all of its derivatives vanish on $|a|+|b|\le\tau/2$. The shaded diamond depicts the local hiding region.}
\label{fig:mechanism}
\end{figure}
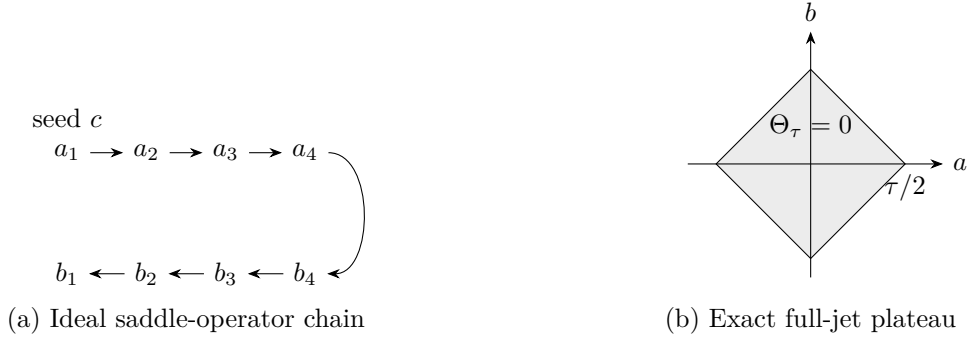

\begin{lemma}[Convex--concave admissibility]
\label{lem:admissibility}
The function \eqref{eq:hardpotential} is globally $\lambda(1-q)$-strongly convex in $x$ and globally $\lambda(1-q)$-strongly concave in $y$.  The saddle operator has the same strong-monotonicity constant. Moreover,
\[
 \Lip(D^p\phi_{U,V})
 \le \lambda C_p^{\rm sm}\tau^{1-p},
\]
where $C_p^{\rm sm}$ depends only on $p$ and the fixed cutoff $\chi$.
\end{lemma}

\begin{proof}
By \hyperref[lem:gate]{Lemma~\ref*{lem:gate}} and the disjointness of the modules,
\[
 \|\nabla^2\psi(r)\|_{\op}\le q.
\]
Hence every principal block also has operator norm at most $q$.  Therefore
\[
 \nabla^2_{xx}\phi\succeq\lambda(1-q)I,
 \qquad
 -\nabla^2_{yy}\phi\succeq\lambda(1-q)I.
\]
The cross terms cancel in the symmetric part of the saddle Jacobian, giving strong monotonicity.

For smoothness, each $(p+1)$st derivative tensor is supported on one of the mutually orthogonal one- or two-dimensional module subspaces.  If $k=p+1\ge3$, then for unit vectors $h_1,\ldots,h_k$, H\"older's inequality gives
\[
 \sum_j\prod_{\ell=1}^k\|P_jh_\ell\|
 \le
 \prod_{\ell=1}^k
 \left(\sum_j\|P_jh_\ell\|^k\right)^{1/k}
 \le1.
\]
\hyperref[lem:gate]{Lemma~\ref*{lem:gate}}(e) then yields the asserted dimension-free bound.
\end{proof}

The saddle operator has a particularly simple form in the $r$ coordinates.  Let
\[
 \overline F=F_\phi/\lambda.
\]
Using \eqref{eq:signature-reversal},
\begin{equation}
 W^T\overline F(z)
 =r-R_m\nabla\psi(r)
 =r-T(r).
 \label{eq:Fchain}
\end{equation}
The map $T=R_m\nabla\psi$ is $q$-Lipschitz.  Its components obey
\begin{equation}
 T_1(r)=c,
 \qquad
 |T_j(r)-q r_{j-1}|\le q\tau,
 \quad j=2,\ldots,2n.
 \label{eq:approxshift}
\end{equation}
The middle relation uses $H_\tau'=h_\tau$; the others use the gate's derivative bounds.

The contraction representation also locates the saddle point. The estimate fits the hard instance inside fixed-radius product balls and leaves room for feasible gap witnesses.

\begin{lemma}[Interior saddle point]
\label{lem:saddlepoint}
The equation $r=T(r)$ has a unique solution $r^\star$ satisfying
\[
 |r_j^\star|\le c+2n\tau,
 \qquad j=1,\ldots,2n,
\]
and therefore
\[
 \|r^\star\|\le\sqrt{2n}(c+2n\tau).
\]
All inactive coordinates of the unique unconstrained saddle point are zero.
\end{lemma}

\begin{proof}
Existence and uniqueness follow from the contraction theorem.  Since $r_1^\star=c$, relation \eqref{eq:approxshift} and induction give
\[
 |r_j^\star|\le q|r_{j-1}^\star|+q\tau\le c+(j-1)\tau.
\]
\end{proof}

\subsection{The Ideal Chain and Error Separation}\label{sec:ideal}

The error mechanism is most transparent before gating. In the ideal chain, a telescoping identity ties the middle coordinate to the seed at the first coordinate. We use that identity first for the residual and then, through a curvature calculation, for a direct saddle-gap witness.

Replace $h_\tau(s)$ by $s$, $H_\tau(s)$ by $s^2/2$, and $\Theta_\tau(a,b)$ by $ab$.  Thus
\begin{equation}
 \psi_0(r)
 =c r_{2n}
 +q\sum_{i=1}^{n-1}r_i r_{2n-i}
 +\frac q2r_n^2,
 \label{eq:psi0}
\end{equation}
and let $\phi_0$ be defined as in \eqref{eq:hardpotential} with $\psi_0$ in place of $\psi$.
In the chain coordinates,
\begin{equation}
 W^T F_{\phi_0}(z)/\lambda=e,
 \qquad
 e_1=r_1-c,
 \qquad
 e_j=r_j-q r_{j-1},\quad j\ge2.
 \label{eq:idealresidual}
\end{equation}

Define
\begin{equation}
 w=(q^{n-1},q^{n-2},\ldots,q,1)\in\R^n.
 \label{eq:w}
\end{equation}
Then
\begin{equation}
 \sum_{j=1}^n w_j e_j=r_n-q^{n-1}c.
 \label{eq:telescoping}
\end{equation}
Since $q=1-1/(4n)$, Bernoulli's inequality gives
\begin{equation}
 q^{n-1}\ge 1-\frac{n-1}{4n}\ge\frac34.
 \label{eq:qpower}
\end{equation}

\begin{lemma}[Ideal residual obstruction]
\label{lem:idealresidual}
If $|r_n|\le\rho$, then
\[
 \|F_{\phi_0}(z)\|
 \ge
 \lambda\frac{q^{n-1}c-\rho}{\sqrt n}.
\]
\end{lemma}

\begin{proof}
By \eqref{eq:telescoping}, Cauchy--Schwarz, and $\|w\|\le\sqrt n$,
\[
 \|e\|\ge\|e_{1:n}\|
 \ge\frac{|w^Te_{1:n}|}{\|w\|}
 \ge\frac{q^{n-1}c-\rho}{\sqrt n}.
\]
Since $W^TW=I$, the active-coordinate projection is contractive, giving
\[
\|F_{\phi_0}(z)\|\ge\|W^TF_{\phi_0}(z)\|=\lambda\|e\|.
\]
\end{proof}

The gap obstruction is stronger than what follows by squaring the Euclidean residual.  Its proof uses the exact curvature energy of the telescoping vector.

Let
\[
 A=\nabla^2_{xx}(\phi_0/\lambda),
 \qquad
 D=-\nabla^2_{yy}(\phi_0/\lambda).
\]
Define the active-coordinate Hessians $A_U:=U^TAU\in\R^{n\times n}$ and $D_V:=V^TDV\in\R^{n\times n}$. In these coordinates,
\begin{equation}
 A_U+D_V=
 \begin{pmatrix}
 2&-q&&&\\
 -q&2&-q&&\\
 &\ddots&\ddots&\ddots&\\
 &&-q&2&-q\\
 &&&-q&2
 \end{pmatrix}.
 \label{eq:AplusD}
\end{equation}
The terminal term $q a_n^2/2$ contributes equally to the $xx$ and $yy$ Hessians and therefore cancels in $A_U+D_V$.  Since $w_i=q w_{i+1}$,
\begin{equation}
 w^T(A_U+D_V)w
 =2\sum_{i=1}^nw_i^2-2q\sum_{i=1}^{n-1}w_iw_{i+1}
 =2.
 \label{eq:energy2}
\end{equation}

\begin{lemma}[Exact ideal gap witness]
\label{lem:idealgap}
Suppose $s=q^{n-1}c-r_n>0$.  Put
\[
 \bar u=Uw,
 \qquad
 \bar v=Vw,
 \qquad
 \Delta_x=\frac{s}{\sqrt2}\bar u,
 \qquad
 \Delta_y=\frac{s}{\sqrt2}\bar v.
\]
Whenever $x+\Delta_x\in\cX$ and $y+\Delta_y\in\cY$,
\begin{equation}
 \Gap_{\phi_0}(x,y)
 \ge
 \phi_0(x,y+\Delta_y)-\phi_0(x+\Delta_x,y)
 =\frac\lambda2s^2.
 \label{eq:exactgapwitness}
\end{equation}
\end{lemma}

\begin{proof}
Write
\[
 f_x=\nabla_x(\phi_0/\lambda),
 \qquad
 f_y=-\nabla_y(\phi_0/\lambda).
\]
The first $n$ chain coordinates of $W^T(f_x,f_y)$ equal $e_{1:n}$, whence
\[
 \langle Uw,f_x\rangle+\langle Vw,f_y\rangle
 =\sqrt2\,w^Te_{1:n}
 =-\sqrt2\,s.
\]
Exact quadratic Taylor expansion gives
\begin{align*}
 &\frac1\lambda\bigl[\phi_0(x,y+\Delta_y)-\phi_0(x+\Delta_x,y)\bigr]\\
 &\quad=-\langle f_y,\Delta_y\rangle-\langle f_x,\Delta_x\rangle
 -\frac12\Delta_x^TA\Delta_x-\frac12\Delta_y^TD\Delta_y\\
 &\quad=s^2-\frac{s^2}{4}w^T(A_U+D_V)w
 =\frac12s^2,
\end{align*}
where the last equality is \eqref{eq:energy2}.
\end{proof}

\subsection{Stability Under Exact Gating}\label{sec:robustness}

The ideal witnesses identify the desired error scale. We next transfer that scale to the smooth gated construction: a uniform potential estimate controls the gap perturbation, and an approximate telescoping identity controls the residual.

\begin{lemma}[Uniform potential approximation]
\label{lem:uniformapprox}
If $\|x\|,\|y\|\le R$, then
\begin{equation}
 |\psi(W^Tz)-\psi_0(W^Tz)|
 \le 3\tau R\sqrt n.
 \label{eq:uniformapprox}
\end{equation}
Consequently,
\[
 |\phi(z)-\phi_0(z)|\le 3\lambda\tau R\sqrt n.
\]
\end{lemma}

\begin{proof}
\hyperref[lem:gate]{Lemma~\ref*{lem:gate}} gives
\[
 |\psi-\psi_0|
 \le q\tau\left(\sum_{i=2}^n|b_i|+|a_n|\right).
\]
The active-frame projection is contractive and the coordinate rotation preserves norms, so
\[
 \|a\|^2+\|b\|^2\le\|x\|^2+\|y\|^2\le2R^2.
\]
Cauchy--Schwarz yields \eqref{eq:uniformapprox}.
\end{proof}

The gated residual obeys the same telescoping inequality up to the accumulated gate error.

\begin{lemma}[Gated residual obstruction]
\label{lem:gatedresidual}
Let $e=W^TF_\phi(z)/\lambda$.  If $|r_n|\le\rho$, then
\begin{equation}
 \|F_\phi(z)\|
 \ge
 \lambda\frac{q^{n-1}c-\rho-n\tau}{\sqrt n}.
 \label{eq:gatedresidual}
\end{equation}
\end{lemma}

\begin{proof}
Write
\[
 \delta_j=T_j(r)-q r_{j-1},
 \qquad |\delta_j|\le q\tau,
 \quad j\ge2.
\]
The recursion $r_1=c+e_1$ and
$r_j=q r_{j-1}+\delta_j+e_j$ gives
\[
 r_n=q^{n-1}c+\sum_{j=1}^nw_je_j
 +\sum_{j=2}^nw_j\delta_j.
\]
The result follows from $\|w\|\le\sqrt n$.
\end{proof}

\subsection{Parameter Choice and Product-Domain Geometry}\label{sec:geometry}

Let
\begin{equation}
 \cX=B_{d_x}(0,R),
 \qquad
 \cY=B_{d_y}(0,R),
 \qquad
 R=\frac{D_Z}{2\sqrt2}.
 \label{eq:domain}
\end{equation}
Then $\operatorname{diam}(\cX\times\cY)=D_Z$.
Choose universal constants
\[
 0<\eta\ll\kappa\ll1
\]
small enough for all inequalities below; for example, one may take
$\kappa=2^{-10}$ and $\eta=2^{-24}$.  Set
\begin{equation}
 c=\frac{\kappa R}{\sqrt n},
 \qquad
 \tau=\frac{\eta c}{n},
 \qquad
 \lambda=\frac{L_p\tau^{p-1}}{C_p^{\rm sm}},
 \label{eq:parameters}
\end{equation}
where $C_p^{\rm sm}$ is the smoothness constant from \hyperref[lem:admissibility]{Lemma~\ref*{lem:admissibility}}.  Then
\[
 \Lip(D^p\phi)\le L_p.
\]
\hyperref[lem:saddlepoint]{Lemma~\ref*{lem:saddlepoint}} and the smallness of $\kappa,\eta$ give
\begin{equation}
 \|x^\star\|,\|y^\star\|
 \le\|r^\star\|
 \le\sqrt2\kappa(1+2\eta)R
 \le\frac R8.
 \label{eq:saddleinterior}
\end{equation}
Thus the unique unconstrained saddle point is an interior saddle point of \eqref{eq:domain}.

The parameter choice fixes the smoothness and diameter simultaneously. It remains to extend the residual obstruction from interior points to the boundary, where the accuracy criterion includes normal vectors. The interior location of the saddle point supplies the needed control on those boundary corrections.

\begin{lemma}[Tangent-residual obstruction on product balls]
\label{lem:normal}
Assume $|r_n|\le\tau$.  Then
\begin{equation}
 \rtan^\phi(z)
 \ge c_0\lambda\frac{c}{\sqrt n}
 \label{eq:tangentfinalscale}
\end{equation}
for a universal constant $c_0>0$.
\end{lemma}

\begin{proof}
If both $x$ and $y$ are in the interiors of their balls, the normal cone is zero, and \hyperref[lem:gatedresidual]{Lemma~\ref*{lem:gatedresidual}}, \eqref{eq:qpower}, and $n\tau=\eta c$ imply
\[
 \rtan^\phi(z)=\|F_\phi(z)\|
 \ge\lambda\frac{c}{2\sqrt n}.
\]

Suppose one block is on its boundary.  Let $\mu=\lambda(1-q)=\lambda/(4n)$.  For every $v\in N_\cZ(z)$, strong monotonicity and the interior zero $F(z^\star)=0$ give
\[
 \langle F(z)+v,z-z^\star\rangle
 \ge\mu\|z-z^\star\|^2.
\]
Hence
\[
 \|F(z)+v\|\ge\mu\|z-z^\star\|.
\]
By \eqref{eq:saddleinterior}, a boundary point has
$\|z-z^\star\|\ge7R/8$.  Therefore
\[
 \rtan^\phi(z)\ge\frac{7\lambda R}{32n}
 \ge c_0\lambda\frac{c}{\sqrt n}
\]
for sufficiently small universal $\kappa$.
\end{proof}

The gap estimate requires feasible primal and dual comparison points. In the inner region, the ideal witness remains inside the balls and survives the gating perturbation; near the boundary, strong convexity or concavity supplies the required separation.

\begin{lemma}[Product-domain gap obstruction]
\label{lem:gapfinal}
Assume $|r_n|\le\tau$.  Then
\begin{equation}
 \Gap_\phi(z)\ge c_1\lambda c^2
 \label{eq:gapfinalscale}
\end{equation}
for a universal constant $c_1>0$.
\end{lemma}

\begin{proof}
Define the terminal chain deficit by
\[
 s=q^{n-1}c-r_n.
\]
By \eqref{eq:qpower}, $s\ge c/2$.  The witness steps in \hyperref[lem:idealgap]{Lemma~\ref*{lem:idealgap}} satisfy
\[
 \|\Delta_x\|=\|\Delta_y\|
 =\frac{s\|w\|}{\sqrt2}
 \le \frac{(c+\tau)\sqrt n}{\sqrt2}
 \le \kappa R,
\]
where the last inequality follows from the chosen small constants.

First suppose
\[
 \|x\|,\|y\|\le R-2\kappa R.
\]
Then $x+\Delta_x\in\cX$ and $y+\Delta_y\in\cY$.  \hyperref[lem:idealgap]{Lemma~\ref*{lem:idealgap}} and \hyperref[lem:uniformapprox]{Lemma~\ref*{lem:uniformapprox}} imply
\begin{align*}
 \Gap_\phi(z)
 &\ge \phi(x,y+\Delta_y)-\phi(x+\Delta_x,y)\\
 &\ge \frac\lambda2s^2-6\lambda\tau R\sqrt n\\
 &\ge \frac\lambda8c^2
     -18\lambda\frac{\eta}{\kappa}c^2
 \ge c_1\lambda c^2,
\end{align*}
where the final inequality follows from $\eta\ll\kappa$.

Now suppose one block lies in the outer shell; say
$\|x\|>R-2\kappa R$.  By \eqref{eq:saddleinterior},
$\|x-x^\star\|\ge R/2$.  Strong convexity--concavity and the saddle-point witnesses $(x^\star,y^\star)$ give
\begin{align*}
 \Gap_\phi(x,y)
 &\ge \phi(x,y^\star)-\phi(x^\star,y)\\
 &\ge \frac\mu2\bigl(\|x-x^\star\|^2+\|y-y^\star\|^2\bigr)\\
 &\ge\frac{\lambda R^2}{32n}
 =\frac{\lambda c^2}{32\kappa^2}
 \ge c_1\lambda c^2.
\end{align*}
The case of an outer-shell dual block is identical.
\end{proof}

Substituting \eqref{eq:parameters} into \eqref{eq:tangentfinalscale} and \eqref{eq:gapfinalscale} gives
\begin{align}
 \lambda\frac{c}{\sqrt n}
 &=\Omega_p\!\left(
  \frac{L_pR^p}{n^{(3p-1)/2}}
 \right),
 \label{eq:rate-tan}\\
 \lambda c^2
 &=\Omega_p\!\left(
  \frac{L_pR^{p+1}}{n^{(3p-1)/2}}
 \right).
 \label{eq:rate-gap}
\end{align}
Because $R=D_Z/(2\sqrt2)$ and $n=N+1$, these are exactly the desired scales.

\section{Information Lower Bounds}\label{sec:information}
The geometric estimates reduce the lower-bound theorem to keeping the final chain coordinate small at the algorithm's output. Exact flatness provides this information barrier for complete scalar jets. We first select hidden directions against a deterministic transcript, then use a random frame to obtain an instance chosen independently of a randomized algorithm's seed.

\subsection{Deterministic transcripts}\label{sec:det}
An adaptive choice of orthogonal directions reveals at most one new frame pair per query. After the last query, one further pair can be chosen orthogonal to the output, while all previous answers remain consistent with the completed potential.

\begin{proposition}[Exact deterministic transcript]
\label{prop:dettranscript}
Let a deterministic algorithm make $N$ feasible queries and then output an arbitrary transcript-dependent point.  In dimensions
$d_x=d_y=2N+2$, one can choose $U,V$ online so that
\begin{enumerate}[label=(\roman*),leftmargin=2em]
\item every oracle answer is the complete scalar $p$-jet of the single final potential \eqref{eq:hardpotential}; and
\item at the final output,
\[
 a_n=b_n=0.
\]
\end{enumerate}
\end{proposition}

\begin{proof}
If the algorithm stops early, pad the transcript with feasible queries whose answers the saved output ignores. At round $t=1,\ldots,N$, after seeing $(x_t,y_t)$, choose
\[
 u_t\perp\Span\{u_1,\ldots,u_{t-1},x_1,\ldots,x_t\},
\]
\[
 v_t\perp\Span\{v_1,\ldots,v_{t-1},y_1,\ldots,y_t\}.
\]
The dimension count makes this possible.  Return the jet of the base quadratic, the seed, and the modules indexed $1,\ldots,t-1$.  Module $t-1$ may reveal the newly chosen pair $(u_t,v_t)$; this is allowed.  Every later module has both of its arguments equal to zero at every query seen so far, hence is exactly flat by \hyperref[lem:gate]{Lemma~\ref*{lem:gate}}(b).

After the algorithm outputs $(\widehat x,\widehat y)$, choose the final pair $(u_n,v_n)$ orthogonal to all previous frame vectors, all query blocks, and the corresponding output block.  There are at most $2N+1$ vectors to avoid in each space, while the dimension is $2N+2$.  Complete the potential with the remaining module and the terminal term.  Every earlier response equals the jet of this final global potential, and
$\langle u_n,\widehat x\rangle=\langle v_n,\widehat y\rangle=0$, which gives $a_n=b_n=0$.
\end{proof}

\hyperref[thm:main]{Theorem~\ref*{thm:main}}(i) follows from \hyperref[prop:dettranscript]{Proposition~\ref*{prop:dettranscript}}, \hyperref[lem:normal]{Lemma~\ref*{lem:normal}} and \hyperref[lem:gapfinal]{Lemma~\ref*{lem:gapfinal}}, and identities \eqref{eq:rate-tan}--\eqref{eq:rate-gap}.

\subsection{Randomized transcripts}\label{sec:rand}

For a randomized policy, the hidden frame is fixed before any query. Concentration in a sufficiently large dimension keeps its unrevealed coordinates inside the gate's flat region throughout a coupled transcript. To implement this argument, sample $U$ and $V$ independently and uniformly from the Stiefel manifolds of $n$ orthonormal columns in $\R^{d_x}$ and $\R^{d_y}$ before the run.

\begin{lemma}[Adaptive Haar-frame concentration]\label{lem:haar}
Let $U$ and $V$ be independent Haar orthonormal frames with $n\ge2$ columns in dimensions $d_x,d_y$, independent of a seed $\xi$. Set
\[
 \mathcal F_{t-1}=\sigma(\xi,u_1,\ldots,u_{t-1},v_1,\ldots,v_{t-1}).
\]
Suppose $(x_t,y_t)\in B_{d_x}(0,R)\times B_{d_y}(0,R)$ is $\mathcal F_{t-1}$-measurable for $1\le t\le n-1$, and the final output $(x_{\rm out},y_{\rm out})$ is $\mathcal F_{n-1}$-measurable and belongs to $B_{d_x}(0,R)\times B_{d_y}(0,R)$ almost surely. If $s>0$, $0<\delta<1$, and
\begin{equation}\label{eq:dimcondition}
 \min\{d_x,d_y\}-n\ge C\frac{R^2}{s^2}\log\frac{n+1}{\delta},
\end{equation}
then, with probability at least $1-\delta$, all inner products $|\langle u_i,x_t\rangle|$ and $|\langle v_i,y_t\rangle|$ for $t\le i\le n$ are at most $s$, as are $|\langle u_n,x_{\rm out}\rangle|$ and $|\langle v_n,y_{\rm out}\rangle|$.
\end{lemma}
\begin{proof}
Conditioned on $\mathcal F_{t-1}$, each future primal column is marginally uniform on the unit sphere orthogonal to the revealed primal columns; independence gives the corresponding statement for the dual frame. For a fixed vector of norm at most $R$, spherical concentration yields
\[
 \Pr\{|\langle u,x\rangle|>s\}\le 2\exp\!\left(-c(d_x-n)s^2/R^2\right).
\]
If $s\ge R$, the event is automatic. Otherwise apply the bound conditionally to each query, then average over the joint history. Apply it once more to the output conditioned on $\mathcal F_{n-1}$. There are fewer than $2n^2+2$ primal and dual events. A union bound, with a sufficiently large absolute $C$, proves the result. Throughout this argument, conditioning uses only the revealed columns and the algorithm's seed.
\end{proof}

After padding early stops, define the round-$t$ potential for $1\le t\le N=n-1$ by
\begin{equation}\label{eq:truncated-potential}
 \phi^{[t]}_{U,V}(x,y)
 =\lambda\left[
 \frac12\|x\|^2-\frac12\|y\|^2-cb_1
 -q\sum_{i=1}^{t-1}\Theta_\tau(a_i,b_{i+1})
 \right],
\end{equation}
with an empty sum when $t=1$. This retains the public quadratic, the seed, and the first $t-1$ two-variable modules, while omitting the terminal term. Run the algorithm with response $\cJ_p\phi^{[t]}_{U,V}(z_t)$ at round $t$. Each response depends only on frame pairs $1,\ldots,t$, so induction makes the virtual query $z_t$ measurable with respect to $\mathcal F_{t-1}$ and the final virtual output measurable with respect to $\mathcal F_{n-1}$. The total policy convention in \hyperref[u:sec:setup]{Section~\ref*{u:sec:setup}} ensures that all these queries and the output remain feasible on every history. Thus \hyperref[lem:haar]{Lemma~\ref*{lem:haar}} applies with
\[
 s=\frac{\tau}{4\sqrt2}.
\]
On the event of \hyperref[lem:haar]{Lemma~\ref*{lem:haar}}, every future chain coordinate at round $t$ has magnitude at most $\tau/4$.  Hence every omitted two-variable module satisfies
\[
 |a_i|+|b_{i+1}|\le\tau/2
\]
and every omitted terminal module satisfies $|a_n|\le\tau/4$.  \hyperref[lem:gate]{Lemma~\ref*{lem:gate}}(b) implies that all omitted values and derivatives vanish exactly.  Inductively, equality of the first $t-1$ responses gives equality of the real and virtual round-$t$ queries. Exact flatness then gives equality of their round-$t$ responses. Thus the full and truncated scalar $p$-jet transcripts coincide round by round.  The real and virtual queries coincide, and the output satisfies
\begin{equation}
 |a_n|,|b_n|\le\tau/4.
 \label{eq:outputslab}
\end{equation}

Because
\[
 \frac{R}{\tau}
 =\Theta\!\left(n^{3/2}\right)
\]
with universal hidden constants, condition \eqref{eq:dimcondition} is met by
\begin{equation}
 d_x,d_y=O(n^3\log n).
 \label{eq:randdim}
\end{equation}
Taking $\delta=1/4$, \hyperref[lem:normal]{Lemma~\ref*{lem:normal}} and \hyperref[lem:gapfinal]{Lemma~\ref*{lem:gapfinal}} hold on the coupling event.  Thus, under the joint distribution of the random frame and the algorithm's seed, the algorithm fails the corresponding target with probability at least $3/4$.  Averaging over frames produces one fixed frame for which the failure probability over the algorithm's seed is at least $3/4$, which is strictly larger than the failure probability $1/3$ allowed by a $2/3$-success guarantee.  This proves \hyperref[thm:main]{Theorem~\ref*{thm:main}}(ii).

\section{Numerical Experiments}\label{u:sec:experiments}
We examine oracle-cost scaling on a finite collection of convex--concave
minimax problems. For each derivative order $p\in\{2,3\}$ and index $n$,
the collection contains an ordered chain and a planar rotation problem:
\[
 \mathcal P_{p,n}=\{C_{p,n},R_{p,\sigma_n}\}.
\]
We measure primal--dual gap, whose complexity is characterized by \hyperref[cor:complexity]{Corollary~\ref*{cor:complexity}}. Each member has the feasible domain, dimension, and initialization specified below.
For a method $A$, let $N_A(f,\epsilon)$ denote its original full-jet
evaluation count until an evaluated point first attains
$\Gap_f\le\epsilon$. We compare methods using the maximum cost
over the same collection at a common target:
\begin{equation}\label{u:eq:experiment-max}
 M_A(p,n)=\max_{f\in\mathcal P_{p,n}}N_A(f,\epsilon_n)
 =\max\{N_A(C_{p,n},\epsilon_n),
         N_A(R_{p,\sigma_n},\epsilon_n)\}.
\end{equation}
Thus each plotted observation summarizes both problems. Panels~(a)
and~(b) use the cubic collection with $p=2$ and the quartic collection
with $p=3$, respectively. The members probe propagation along a chain
and progress through a nearly bilinear region. Both use conservative
smoothness inputs $L_1=L_p=1$ and complete scalar jets through order $p$.

\paragraph{Ordered-chain member.}
Let $h=n^{-1/2}$ and define
\[
 X_n=\{x\in\R^n:0\le x_n\le\cdots\le x_1\le h\},
 \qquad Y_n=[0,h]^n.
\]
With $x_0=h$, the chain objective is
\begin{equation}\label{u:eq:experiment-chain}
 C_{p,n}(x,y)=c_p\sum_{i=1}^n y_i(x_{i-1}-x_i)^p,
 \qquad c_2=\frac1{32},\quad c_3=\frac1{128}.
\end{equation}
It is convex in $x$ on the ordered domain and linear in $y$.
Both marginal domains have Euclidean diameter one. The initialization
is $(x,y)=(0,0)$, and the exact primal--dual gap is
$\Gap(x,y)=c_ph\sum_i(x_{i-1}-x_i)^p$.

\paragraph{Rotation member.}
The second member of $\mathcal P_{p,n}$ is defined on $[-1/2,1/2]^2$ by
\begin{equation}\label{u:eq:experiment-rotation}
 R_{p,\sigma}(u,v)=g_p(u)+\sigma uv-g_p(v),
 \qquad g_p(t)=\frac{(|t|-0.45)_+^{p+1}}{(p+1)!},
\end{equation}
where $(s)_+=\max\{s,0\}$. We start from $(u,v)=(0.4,0.4)$.
The objective is bilinear in its central region; the outer convex
terms provide higher-order curvature. Here
$\Gap_f(x,y)=\max_{v\in Y}f(x,v)-\min_{u\in X}f(u,y)$.

\paragraph{Accuracy and parameter scaling.}
For chain lengths $n=6,\ldots,16$, we set
\[
 \epsilon_n=\frac{1}{512n^{5/2}}\quad(p=2),\qquad
 \epsilon_n=\frac{1}{2048n^4}\quad(p=3),\qquad
 \sigma_n=(10\epsilon_n)^{p/(p+1)}.
\]
The accuracy schedule is tied to the chain geometry: on the face
$x_n=0$, Jensen's inequality gives
$\Gap(x,y)\ge c_p n^{-(3p-1)/2}$, and we set the target to one
sixteenth of this value. The rotation schedule varies the coupling
with the same accuracy target. Together, these choices specify a sweep
over problem instances indexed by $n$, with a shared target for each pair.

\paragraph{Methods and measurement.}
We compare numerical implementations of higher-order extragradient
(Tensor EG~\citep{monteiro2012,adil2022}),
the AIPE approach of \citet{chen2025,chen2026}, and our scalar implementation of the Halpern-based
VI upper-bound method of \citet{zhang2026matching} (Ours). All implementations use finite-precision
nested solvers.
The latter uses radius-growth factors $2$ and $1.5$ for $p=2$ and
$p=3$, respectively; the third-order factor was selected in an earlier
parameter study and held fixed for this sweep.
All methods use the same initialization and target on each instance.
We record the number of original full-jet evaluations until an evaluated
point first satisfies $\Gap\le\epsilon_n$, using the exact gap formulas
as an external diagnostic. Arithmetic on a model constructed from stored jets is free.
All displayed runs attained their targets. We plot
$M_A(p,n)/M_A(p,6)$ against $\epsilon_6/\epsilon_n$, using the collection
maximum in \eqref{u:eq:experiment-max}. This normalization compares
relative query-cost growth across the problem collection.

\begin{figure}[t]
 \centering
 \includegraphics[width=\textwidth]{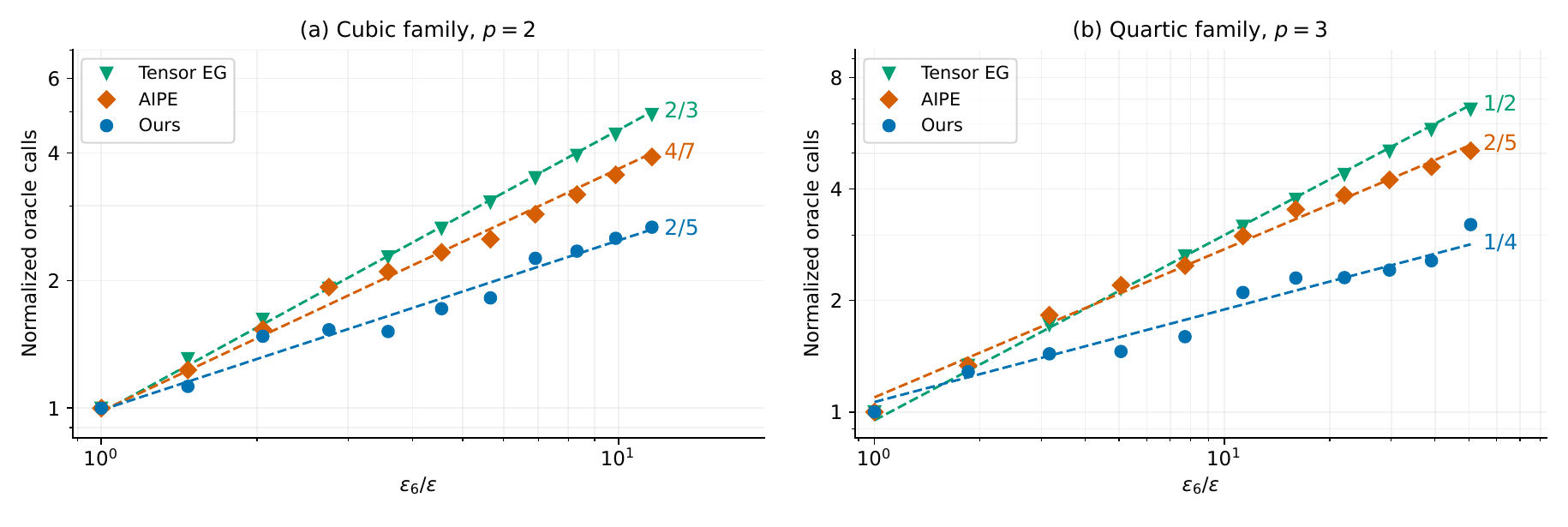}
 \caption{Oracle-cost scaling on (a) the cubic family with $p=2$ and
 (b) the quartic family with $p=3$. Markers show normalized measured
 first-hit costs, maximized over the chain and rotation problems.
 Dashed lines have fixed theoretical exponents, labeled at their right
 endpoints, with logarithmic factors omitted. Their multiplicative
 constants are fitted to all observations in log space; they are
 theoretical-slope reference curves for the measured costs.}
 \label{u:fig:experiment-rates}
\end{figure}

\paragraph{Results and interpretation.}
\hyperref[u:fig:experiment-rates]{Figure~\ref*{u:fig:experiment-rates}} shows slower relative cost growth for
Ours than for the two baselines in these sweeps. The empirical log--log
slopes for Tensor EG, AIPE, and Ours are approximately
$(0.645,0.541,0.395)$ for $p=2$ and $(0.480,0.409,0.270)$ for $p=3$.
The corresponding theoretical reference exponents are
$(2/3,4/7,2/5)$ and $(1/2,2/5,1/4)$.
The reference-line constants are chosen by minimizing the squared
logarithmic residuals while keeping these exponents fixed.

Across both problem collections, the observed query-cost growth is
consistent with the ordering of the theoretical reference exponents.
Ours has the smallest empirical slope in each panel, with $0.395$ and
$0.270$ close to the respective reference values $2/5$ and $1/4$ over
the tested accuracy ranges.

\section{Conclusion}
We establish the lower-bound exponent $2/(3p-1)$ for arbitrary adaptive deterministic and randomized algorithms in the complete scalar higher-order oracle model. The construction combines scalar integrability, dimension-independent smoothness, and exact local hiding, while direct product-domain witnesses treat both tangent residual and saddle gap. Together with the VI upper bound of \citet{zhang2026matching}, this determines the high-dimensional minimax query complexity up to logarithmic factors, with constants depending only on $p$. The characterization identifies the information required to solve a scalar saddle problem under arbitrary adaptive query rules. Numerical comparisons at orders $p=2$ and $p=3$ exhibit relative oracle-cost growth consistent with the theoretical accuracy exponents.

\bibliographystyle{unsrtnat}
\setlength{\bibsep}{3pt}
\bibliography{references}
\end{document}